\documentclass[12pt,a4paper]{article}

\usepackage{amsthm,amsmath,amssymb}
\usepackage{enumerate}
\usepackage[margin=3cm,a4paper]{geometry}
\numberwithin{equation}{section}
\usepackage{graphics,graphicx}
\usepackage{color}

\newtheorem{theorem}{Theorem}[section]

\newtheorem{THM}[theorem]{Theorem}
\newtheorem{LEM}[theorem]{Lemma}

\newtheorem{COR}[theorem]{Corollary}

\newtheorem{QUES}[theorem]{Question}
\theoremstyle{definition}
\newtheorem*{DEF}{Definition}

\newcommand{\pr}{\mathbb{P}}
\newcommand{\ex}{\mathbb{E}}
\newcommand{\indi}{\mathbf{1}}

\newcommand{\rhoo}{\rho_{_{\! 0}}}
\newcommand{\Hom}{\mathrm{Hom}}

\newcommand{\old}[1]{}
\newcommand{\beq}[1]{\begin{equation}\label{#1}}
\newcommand{\eeq}{\end{equation}}

\newcommand{\bb}{\Big}

\newcommand{\bean}{\begin{eqnarray*}}
\newcommand{\eean}{\end{eqnarray*}}
\newcommand{\eps}{\varepsilon}
\newcommand{\bea}[1]{\begin{eqnarray}\label{#1}}
\newcommand{\eea}{\end{eqnarray}}

\newcommand{\ra}{\rightarrow}

\newcommand{\sq}{\mbox{\raisebox{0.2ex}{\tiny{$\, \square \,$}}}}

\begin{document}

\author{
Jeong Han Kim \thanks{School of Computational Sciences,  Korea Institute for Advanced Study (KIAS),  Seoul, South Korea. Email: jhkim@kias.re.kr.~
The author was supported by the National Research Foundation of Korea (NRF) Grant funded by the Korean Government (MSIP) (NRF-2012R1A2A2A01018585) and KIAS internal Research Fund CG046001. This work was partially carried out while the author was visiting Microsoft Research, Redmond and  Microsoft Research, New England.}
\and
Choongbum Lee \thanks{Department of Mathematics,
MIT, Cambridge, MA 02139-4307. Email: cb\_lee@math.mit.edu.}
\and
Joonkyung Lee \thanks{
Mathematical Institute, University of Oxford, OX2 6GG, United Kingdom. 
Email: Joonkyung.Lee@maths.ox.ac.uk. Supported by ILJU Foundation of Education and Culture.
}
}

\title{Two Approaches to Sidorenko's Conjecture}
\date{}

\maketitle

\begin{abstract}
Sidorenko's conjecture states that for every bipartite graph $H$ on $\{1,\cdots,k\}$
\begin{eqnarray*} 
  \int \prod_{(i,j)\in E(H)} h(x_i, y_j) d\mu^{|V(H)|}
			\ge \left( \int h(x,y) \,d\mu^2 \right)^{|E(H)|}
\end{eqnarray*}
holds, where $\mu$ is the Lebesgue measure on $[0,1]$ and 
$h$ is a bounded, non-negative, symmetric, measurable function on $[0,1]^2$.
An equivalent discrete form of the conjecture is 
that the number of homomorphisms from a bipartite graph $H$ 
to a graph $G$ is asymptotically at least the expected 
number of homomorphisms from $H$ to 
the Erd\H{o}s-R\'{e}nyi random graph with the same expected edge density as $G$.
In this paper, we present two approaches to 
the conjecture. 
First, we introduce the notion of tree-arrangeability,
where a bipartite graph $H$ with bipartition $A \cup B$ is
tree-arrangeable if neighborhoods of vertices in $A$ have
a certain tree-like structure. We show that 
Sidorenko's conjecture holds for all 
tree-arrangeable bipartite graphs. 
In particular, this implies that Sidorenko's conjecture
holds if there are two vertices $a_1, a_2$ in $A$ such that each
vertex $a \in A$ satisfies $N(a) \subseteq N(a_1)$
or $N(a) \subseteq N(a_2)$, and also
implies a recent result of Conlon, Fox, and Sudakov \cite{CoFoSu}.
Second, 
if $T$ is a tree
and $H$ is a bipartite graph satisfying Sidorenko's conjecture,
then it is shown that
the Cartesian product $T \sq H$ of $T$ and $H$ also satisfies
Sidorenko's conjecture. This result
implies that, for all $d \ge 2$, the $d$-dimensional grid with
arbitrary side lengths satisfies Sidorenko's conjecture.
\end{abstract}

\section{Introduction}

In this paper, we study a beautiful conjecture of 
Sidorenko \cite{Sidorenko9192} on a correlation inequality related to bipartite graphs.
The conjecture states that, for every bipartite graph $H$ on $\{1,2,\cdots,k\}$,
\begin{eqnarray} 
  \int \prod_{(i,j)\in E(H)} h(x_i, y_j) d\mu^{|V(H)|}
			\ge \left( \int h(x,y) \,d\mu^2 \right)^{|E(H)|}\label{eq:analytic_form}
\end{eqnarray}
holds, where $\mu$ is the Lebesgue measure on $[0,1]$ and 
$h$ is a bounded, non-negative, symmetric, measurable function on $[0,1]^2$.
Throughout the paper, a graph means a simple graph unless specified otherwise. 

Sidorenko \cite{Sidorenko9192, Sidorenko93} noted that 
the functional on the left-hand side of the correlation 
inequality \eqref{eq:analytic_form} 
often appears in various fields
of science: Feynman integrals in quantum field theory \cite{Stell},
Mayer integrals in classical statistical mechanics,
and multicenter integrals in quantum chemistry \cite{DaLeMo}.

The correlation inequality resembles the 
famous FKG inequality \cite{FKG},
which asserts that increasing functions are 
positively correlated when the underlying measure is 
log-supermodular over a finite distributive lattice.
Despite the similarity, 
it is unclear that the FKG inequality can be 
applied to show \eqref{eq:analytic_form}: There exist a function $h$ and 
two edge disjoint subgraphs $H_1$ and $H_2$ of a bipartite graph
such that $\prod_{(i,j)\in E(H_1)} h(x_i,y_j)$ and 
$\prod_{(i,j)\in E(H_2)} h(x_i,y_j)$ are not positively correlated \cite{London}. 
It is unknown  whether every bipartite graph $H$ can be decomposed
into two edge disjoint non-empty subgraphs $H_1$ and $H_2$ (possibly depeding on $h$)
so that $\prod_{(i,j)\in E(H_1)} h(x_i,y_j)$ and 
$\prod_{(i,j)\in E(H_2)} h(x_i,y_j)$ are positively correlated.

\medskip{}

An equivalent discrete form expresses the conjecture 
in terms of graph homomorphisms.   
For two graphs $H$ and $G$, a \emph{homomorphism }from $H$ to $G$
is a mapping $g\,:\, V(H)\rightarrow V(G)$ such that $\{g(v),g(w)\}$
is an edge in $G$ whenever $\{v,w\}$ is an edge in $H$. 
Let $\Hom(H,G)$ denote the set of all homomorphisms
from $H$ to $G$, and let $t_{H}(G)$ be the probability 
that a uniform random mapping from $H$ to $G$ is a homomorphism, 
i.e., $$t_{H}(G)=\frac{|\Hom(H,G)|}{|V(G)|^{|V(H)|}}.$$ 
The discrete form of Sidorenko's conjecture states that 
for every bipartite graph $H$, 
\begin{equation}
t_{H}(G)\ge t_{K_{2}}(G)^{|E(H)|}
~~\text{ for  all graphs $G$.}\label{eq:sido_ineq}
\end{equation}
If $G$ is the Erd\H{o}s-R\'enyi random graph $G(n,p)$,
the mean of $t_H(G(n,p))$ is $p^{|E(H)|}$ plus an error 
term of smaller order of magnitude.
Thus \eqref{eq:sido_ineq} roughly asserts that $t_H(G)$ is minimized when $G$ is a random graph.

In fact, many problems in extremal graph
theory can be expressed using homomorphisms. 
For instance, the chromatic number
$\chi(H)$ is the minimum integer $r$ such that there
exists a homomorphism from $H$ to the complete graph $K_r$ on $r$ vertices.
The problem of finding a copy of $H$ in $G$
can be stated as the problem of finding an injective homomorphism from
$H$ to $G$.
A classical theorem of Tur\'an \cite{Turan}
states that, for all integers $r \ge 3$, 
every graph $G$ with more than  
$\left(1-\frac{1}{r-1}\right)\frac{|V(G)|^2}{2}$
edges contains $K_r$ as a subgraph. 
In terms of graph homomorphisms, 
it may be re-stated as follows: 
For all integers $r\ge 3$, if $t_{K_{2}}(G)>1-\frac{1}{r-1}$, 
then there exists a homomorphism from $K_{r}$ to $G$.
Since the alternative definition of $\chi(H)$ given above implies
that there exists a homomorphism from $H$ to $K_{\chi(H)}$,
it then follows that for all graphs $H$, there
exists a homomorphism from $H$ to $G$ whenever
$t_{K_2}(G)>1-\frac{1}{\chi(H)-1}$, which also follows from 
Erd\H{o}s-Stone Theorem \cite{ErStone}.
Lov\'asz and Simonovits \cite{LoSi} conjectured 
a kind of generalization
of Tur\'an's theorem in 1983 stating that
for  an integer $r\ge 3$ and  $t_{K_2}(G)=\rhoo$  fixed,
$t_{K_r}(G)\geq F(r,\rhoo)+O(|V(G)|^{-2})$ 
for a certain function $F(r,\rhoo)$ of $r$ and $\rhoo$. 
Razborov \cite{Razborov} proved the conjecture for $r=3$ in 2008,
and Nikiforov \cite{Nikiforov} proved it for $r=4$ in 2011.
Recently, Reiher \cite{Reiher} settled the conjecture for all
values of $r$. The equality holds for some complete
$(s+1)$-partite graphs such that the first $s$ parts are of the same size 
and the last part is not larger than the others.

Erd\H{o}s and Simonovits \cite{ErSi,Simonovits} made a similar 
conjecture for bipartite graphs in 1984. They conjectured
that if $H$ is a bipartite graph, then there exists
a positive constant $c_{_{\! H}}$ depending only on $H$  such that 
\begin{equation}
t_{H}(G)\ge c_{_{\! H}} t_{K_2}(G)^{|E(H)|}\label{eq:Erdos_Simon}
\end{equation} for all graphs $G$ (their conjecture
was originally stated in terms of injective homomorphisms
but is equivalent to this form).
It turns out that this conjecture is equivalent to 
Sidorenko's conjecture, as Sidorenko himself showed 
in \cite{Sidorenko9192} using  a tensor power trick.
Recall that, for a bipartite graph $H$,
Sidorenko's conjecture formalizes the idea that the minimum of $t_H(G)$
over all graphs $G$ of the same $t_{K_2}(G)$ must be attained
when $G$ is the random graph with the same $t_{K_2}(G)$.
This is in contrast with the Lov\'asz-Simonovits's 
conjecture above, which is now Reiher's theorem, 
where the extremal graphs have a
deterministic structure. This may be regarded as an example showing that 
there is a difference between fundamental 
structures of bipartite graphs and complete graphs.

\medskip{}

Sidorenko's conjecture is known to be true 
only for a few bipartite graphs $H$.
We say that a bipartite graph $H$ has \emph{Sidorenko's property}
if (\ref{eq:sido_ineq}) holds for all graphs $G$.
That paths have Sidorenko's property \cite{BlRo, MuSm}
was proved around 1960, earlier than Sidorenko suggested the conjecture.
Sidorenko himself \cite{Sidorenko9192} showed that trees, even cycles, and complete bipartite
graphs have Sidorenko's property. 
He also proved that, for a bipartite graph $H$ 
with bipartition $A\cup B$, $H$ has Sidorenko's property if $|A|\leq 4$.
Recently, Hatami \cite{Hatami} proved that hypercubes have Sidorenko's property
by developing a concept of norming graphs.
He proved that every norming graph has Sidorenko's property, and 
that all hypercubes are norming graphs.
Conlon, Fox, and Sudakov \cite{CoFoSu} proved
that if $H$ is a bipartite graph with a bipartition $A\cup B$
and there is a vertex in $A$ adjacent to all vertices in $B$, 
then $H$ has Sidorenko's property.
Sidorenko \cite{Sidorenko9192} and Li and Szegedy \cite{LiSz}
introduced some recursive processes that construct a new graph
from a collection of graphs so that the new one has Sidorenko's property
whenever all the graphs in the collection have the property.
Li and Szegedy \cite{LiSz} introduced some recursive processes that construct a new graph
from a collection of graphs so that the new one has Sidorenko's property
whenever all the graphs in the collection have the property.
On the other hand, the simplest graph not known to have Sidorenko's property
is $K_{5,5}\setminus C_{10}$, a $3$-regular graph on $10$ vertices.

\medskip{}

In this paper, we further study Sidorenko's conjecture by taking two different approaches. 
The first approach uses normalizations by certain conditional expectations 
and Jensen's inequality for logarithmic functions.
This approach is partly motivated by Li and Szegedy \cite{LiSz}.
For a bipartite graph $H$ with bipartition $A \cup B$, 
$H$ is tree-arrangeable if the family of neighborhoods of vertices in $A$ 
has a certain tree-like structure.
We show that all tree-arrangeable bipartite graphs have Sidorenko's property.
For  instance, if there is a vertex in $A$ adjacent to all vertices in $B$,
then $H$ is tree-arrangeable with a star as the corresponding tree. 
Hence our result generalizes the result of Conlon, Fox, and Sudakov \cite{CoFoSu}.

Second, we develop a recursive procedure that preserves
Sidorenko's property. For two graphs $H_1$ and $H_2$,
let the {\em Cartesian product} $H_1 \sq H_2$ (also known as the {\em box product}) 
be the graph over the vertex set
$V(H_1) \times V(H_2)$ such that two vertices $(u_1,u_2)$ and $(v_1,v_2)$ 
are adjacent if and only if 
(i) $u_1$ and $v_1$ are adjacent in $H_1$ and $u_2=v_2$,
or (ii) $u_2$ and $v_2$ are adjacent in $H_2$ and $u_1=v_1$.
We prove that if $T$ is a tree and $H$ is a bipartite graph 
with Sidorenko's property,
then $T \sq H$ also has Sidorenko's property.

\medskip

To present the result that 
the first approach yields, let $H$ be a bipartite graph with bipartition $A\cup B$. 
For a vertex $u$ of $H$, the neighborhood  of $u$ in $H$ is denoted
by $\Lambda_u$.\footnote{
We intentionally avoid using standard notation $N(a)$ in 
order to clarify that the underlying graph is $H$, as  
two different graphs $H$ and $G$ are concerned.}
An independent set $U$ of $H$ is \emph{$T$-arrangeable}
for a tree $T$ on $U$, if 
\begin{equation}\label{intersection}
\Lambda_{u}\cap\Lambda_{v}=\bigcap_{w\in P}\Lambda_{w}
~~~\mbox{for every path $P$ in $T$ connecting  $u$ and  $v$.}
\end{equation}
We say that $U$ is \emph{tree-arrangeable} if it is
$T$-arrangeable for some tree $T$, and $H$ is
\emph{tree-arrangeable} if there exists a 
bipartition $A\cup B$ of $H$ such that
 $A$ is tree-arrangeable.

For example, an independent set $U$ with 
$|U|\le2$ is trivially tree-arrangeable. 
Thus, if  a bipartite graph $H$ has a bipartition 
$A\cup B$ with $|A|\le 2$, then it is tree-arrangeable.
As another example, let $H$ be a bipartite graph with 
bipartition $A\cup B$. If $H$ has a vertex $a \in A$ adjacent
to all vertices in $B$,
then $A$ is $T$-arrangeable where $T$ is a star on $A$ centered at $a$.
For complete bipartite graphs $H$, 
any tree on $A$ can be used to show 
that $A$, and thus $H$, is tree-arrangeable. 
The last example exhibits the fact that 
the choice of $T$ is not necessarily unique.

The concept of tree-arrangeability
is closely related to tree decompositions \cite{RobertSeymour}, and 
also to Markov Random Field Models used in statistical physics \cite{StatPhy}
and image processing \cite{ImgProcess}. 
This will be discussed more in the concluding remarks.

We show that tree-arrangeable bipartite graphs have Sidorenko's property.
\begin{THM} \label{thm:main_thm}
If a bipartite graph $H$ is tree-arrangeable, then $H$ has Sidorenko's property.
\end{THM}
We have seen that a complete bipartite graph, a 
bipartite graph with bipartition $A\cup B$ and  $|A|\leq 2$, 
and a bipartite graph having a vertex adjacent to
all vertices on the other side are tree-arrangeable.
Therefore, our theorem implies that these graphs have Sidorenko's property.
Another interesting example is a bipartite graph $H$ with bipartition $A\cup B$ and two vertices $a_1,a_2 \in A$
such that $\Lambda_a\subseteq\Lambda_{a_1}$ or $\Lambda_a\subseteq\Lambda_{a_2}$ 
for every $a\in A$.
In this case, we may take a tree $T$ on $A$ such that 
there is an edge connecting $a_1$ and $a_2$, and each
$a\neq a_1,a_2$ in $A$ with $\Lambda_a\subseteq\Lambda_{a_1}$ is a leaf adjacent to $a_1$,
and the other vertices are leaves adjacent to $a_2$.
It is easy to see that $H$ is $T$-arrangeable, and thus $H$ has Sidorenko's property. 
This example does not seem to follow from the recursive procedure introduced by Li and Szegedy \cite{LiSz}.

\medskip{}

The second theorem proves that Sidorenko's property is preserved
under taking Cartesian products with trees, here the Cartesian 
product $H_1 \sq H_2$ of two graphs $H_1$ and $H_2$
is  the graph on $V(H_1)\times V(H_2)$ 
such that two vertices $(u_1,u_2)$ and $(v_1,v_2)$ are adjacent if and only if 
either (i) $u_1$ and $v_1$ are adjacent in $H_1$ and $u_2=v_2$,
or (ii) $u_2$ and $v_2$ are adjacent in $H_2$ and $u_1=v_1$.

\begin{THM} \label{thm:box_product}
If $T$ is a tree and $H$ is a bipartite graph having Sidorenko's property, 
then $T \sq H$ also has Sidorenko's property.
\end{THM}

Since paths of all lengths are known to have Sidorenko's property,
by repeatedly applying Theorem \ref{thm:box_product} with paths
$P_1, P_2, \cdots, P_d$ of various lengths, we obtain that
the $d$-dimensional grid $P_1 \sq P_2 \sq \cdots \sq P_d$
has Sidorenko's property. 
This approach especially  yields a
simple proof of the statement that the hypercube
$K_2 \sq K_2 \sq \cdots \sq K_2$ satisfies Sidorenko's property,
which was first proven by Hatami [10].

\medskip{}

The paper is organized as follows.
The proof of Theorem \ref{thm:main_thm} will be given in Section
\ref{sec:proof_mainthm}.
The proof of Theorem \ref{thm:box_product} and its applications
will be given in Section \ref{sec:recursive}.
In the last section, Section \ref{sec:conclude}, we will 
further discuss tree-arrangeability in the context of
tree decompositions and Markov Random Field Models,
and pose some open problems.


\section{Tree-arrangeable Bipartite Graphs}
\label{sec:proof_mainthm}

In this section, we prove Theorem \ref{thm:main_thm}
using normalizations by certain conditional expectations 
and Jensen's inequality for logarithmic functions.
Recall that $t_H(G)$ represents the probability that the uniform random mapping 
from $V(H)$ to $V(G)$ is a graph homomorphism. 
Let $x: V(H) \rightarrow V(G)$ be a mapping chosen uniformly at random among
all $|V(G)|^{|V(H)|}$ mappings from $V(H)$ to $V(G)$. 
For the sake of simplicity, we write
\[
	x_u := x(u) \quad \text{for } u \in V(H)  \quad \text{and} \quad
	x(\Lambda) := \, \mbox{the sequence} ~(x_u)_{u \in \Lambda} \quad \text{for } \Lambda \subseteq V(H).
\]
As in the previous section,  
$\Lambda_u$ is the set of neighbors of $u$ in $H$, for $u \in V(H)$.
For a bipartite graph $H$ with bipartition $A \cup B$,
we now have that 
\[
	t_H(G)
	=
	\ex\left[	\prod_{a\in A } \indi(x_{a} \sim x(\Lambda_a )\text{ in } G)\right],
\]
where $\indi(x_{a} \sim x(\Lambda_a )\text{ in }G)$ is the indicator random variable of the
event that 
$x_a$ 
is adjacent to all vertices in $x(\Lambda_a)$ in $G$ and $\indi(x_{a} \sim \emptyset \text{ in }G)\equiv 1$.

For a vertex $v$ of $G$, 
it is convenient to consider the degree density 
$\rho_G (v):= \frac{d_{G}(v) }{n}$ 
rather than the degree itself. 
The mean of $\rho_G(v)$ over all $v\in V(G)$ is denoted by $\rhoo(G)$. Then, 
\[ 
   \rhoo(G)=\frac{1}{n}\sum_{v\in V(G)}\rho_G(v)=\frac{2|E(G)|}{n^{2}}=t_{K_{2}}(G).
\]
We will simply write 
$\indi(x_{a} \sim x(\Lambda_a ))$, $\rho(v)$, and $\rhoo$ 
for $\indi(x_{a} \sim x(\Lambda_a )\text{ in }G)$,
$\rho_G(v)$, and $\rhoo(G)$, respectively,
if  no confusion arises.

\medskip\noindent
In these notations, $H$ has Sidorenko's property if 
\beq{JH1}
	\ex\left[	\prod_{a\in A } \indi(x_{a} \sim x(\Lambda_a ))\right] \ge \rhoo^{|E(H)|},
\eeq
or equivalently,
\begin{align} 
	\label{eq:log_target}
	\ln \ex\left[	\prod_{a\in A } \indi(x_{a} \sim x(\Lambda_a ))\right]
	 \ge |E(H)| \ln \rhoo.
\end{align}

\medskip

The case of  $H$ being a star, say centered at $a$, 
is a simple example that  may illustrate
how to show  \eqref{JH1} or \eqref{eq:log_target} using Jensen's inequality. 
Taking $A=\{a\}$, we may have that 
\begin{equation}\label{eq:simple_case}
	\ex	\left[\indi(x_{a} \sim x(\Lambda_a ))\right]
	=\ex\big[\ex	\left[\indi(x_{a} \sim x(\Lambda_a ))\vert x_a\right]\big]
	=\ex\left[\rho(x_a)^{|\Lambda_a|}\right]\geq
\rhoo^{|\Lambda_a|}	=\rhoo^{|E(H)|},
\end{equation}
where the last inequality follows from Jensen's inequality.
For another way to show the inequality, one may normalize the indicator function $\indi(x_{a} \sim x(\Lambda_a ))$ by 
$\rhoo\rho(x_{a})^{|\Lambda_a|-1}$, i.e., 
$$f_a:= \frac{\indi(x_{a} \sim x(\Lambda_a ))}{\rhoo\rho(x_{a})^{|\Lambda_a|-1}},$$
provided $G$ has no isolated vertex. (It will be shown that we can always assume so).
A similar argument to that used in \eqref{eq:simple_case} implies that $f_a$ is normalized: 
\beq{JH10}
\ex[f_a]=\ex\big[\ex[f_a|x_a]\big]
=\ex\left[\frac{\ex[\indi(x_a\sim x(\Lambda_a))|x_a]}{\rhoo\rho(x_a)^{|\Lambda_a|-1}}\right]
=\ex[\rho(x_a)/\rhoo]=1.
\eeq
Since
\begin{equation}
	\ln \ex\big[\indi(x_{a} \sim x(\Lambda_a ))\big]
	=\ln\ex[f_a\rhoo\rho(x_a)^{|\Lambda_a|-1}]\label{eq:log_baby}
\end{equation}
and the logarithmic function is concave, 
Jensen's inequality on the new probability measure 
$\pr^*[\mathcal{E}]=\ex[f_a\indi_{\mathcal{E}}]$
yields
$$
\ln \ex\big[\indi(x_{a} \sim x(\Lambda_a ))\big]\geq
\ex[f_a\ln\left(\rhoo\rho(x_a)^{|\Lambda_a|-1}\right)]
	= \ex[f_a\ln\rhoo]+(|\Lambda_a|-1)\ex[f_a\ln\rho(x_a)].
$$
As $\ex[f_a\ln\rhoo]=\ex[f_a]\ln\rhoo=\ln\rhoo$
and 
$$
\ex[f_a\ln\rho(x_a)]
=\ex\big[\ex[f_a\ln\rho(x_a)|x_a]\big]
=\ex\big[\ex[f_a|x_a]\ln\rho(x_a)\big]
=\rhoo^{-1}\ex[\rho(x_a)\ln\rho(x_a)],
$$ 
the convexity of the function $x\ln x$ gives 
$$
\ex[\rho(x_a)\ln\rho(x_a)]\geq\ex[\rho(x_a)]\ln\ex[\rho(x_a)]=\rhoo\ln\rhoo
$$ 
and thus
$$
	\ln \ex\big[\indi(x_{a} \sim x(\Lambda_a ))\big]
	\geq\ln\rhoo +(|\Lambda_a|-1)\rhoo^{-1}\ex[\rho(x_a)\ln\rho(x_a)]
	\geq |E(H)|\ln\rhoo.
$$
This scheme is motivated by Li and Szegedy \cite{LiSz}.
Though it looks much more complicated than \eqref{eq:simple_case}, 
this approach turns out to be more powerful in proving that 
certain bipartite graphs have Sidorenko's property. 

\medskip


\newcommand{\mn}{\medskip\noindent}

\mn

We first show that it is enough to consider $G$ with no isolated vertex.
\begin{LEM} Let $H$ be a bipartite graph. 
If 
$\, t_H(G) \geq (t_{K_2} (G))^{|E(H)|} $
for all graphs $G$ with no isolated vertex,
then $H$ has Sidorenko's property.
\end{LEM} 

\begin{proof} Let $G$ be a  graph  on $n$ vertices with $k\geq 1$ isolated vertices. Then, for the induced subgraph $H_1$ (resp. $G_1$) of $H$ on the set of all non-isolated vertices in $H$ (resp. $G$),  it follows that 
$$ t_H (G) 
= t_{H_1} (G)
= t_{H_1} (G_1)  \left(\frac{n-k}{n}\right)^{|V(H_1)|}
$$
and 
$$ (t_{K_2} (G))^{|E(H)|}
= (t_{K_2} (G_1))^{|E(H)|} \left(\frac{n-k}{n}\right)^{2|E(H)|}
, $$
where $(\frac{n-k}{n})^{|V(H_1)|}$ is the probability that all vertices in $V(H_1)$ are mapped to non-isolated 
vertices of $G$ and similarly  $(\frac{n-k}{n})^{2}$
is the probability that the two vertices in $V(K_2)$ are mapped to non-isolated 
vertices of $G$. Since 
$t_{H} (G_1)\geq (t_{K_2} (G_1))^{|E(H)|}$ by the hypothesis  and 
$|V(H_1)|\leq 2|E(H_1)|\leq 2|E(H)|$,  
we have that  
$t_H (G)  \geq (t_{K_2} (G))^{|E(H)|}$, as desired. 
\end{proof}

We now assume that $G$ has no isolated vertex. 
To bound 
$$ t_H(G) = 
\ex\Big[	\prod_{a\in A } \indi(x_{a} \sim x(\Lambda_a ))\Big]$$
from below, 
we plan to normalize the indicator random variable 
$\indi(x_{a} \sim x(\Lambda_a ))$ twice,  first by 
$\rhoo \rho(x_{a})^{|\Lambda_a|-1}$ as before, 
and then by a certain conditional expectation. 
In both cases, it is important that we avoid dividing by zero.
Since $G$ has no isolated vertex, the first normalization causes no problem.
The second normalization will be  possible if 
$\frac{\indi(x_{a} \sim x(\Lambda_a ))}{\rhoo\rho(x_{a})^{|\Lambda_a|-1}}$
is not zero, which is unfortunately not true in general. 
Hence, we consider a slight variation of the function.  Namely, for a vertex $u$ of $H$, 
\[
  f_{u}=f_{\varepsilon, u} = \frac{\indi(x_{u} \sim x(\Lambda_u ))+\varepsilon\rho(x_u)^{|\Lambda_u|}}{\rhoo\rho(x_{u})^{|\Lambda_u|-1}},
\]
where $\varepsilon>0$  will go to $0$. The term involving $\varepsilon $ is purely technical 
to make the second normalizing factor below non-zero. 
Though it is a slight abuse of notation, we have written 
$f_u$ for $f_{\varepsilon, u}$ for the sake of simplicity. 

The second normalization requires some notations: 
Let $H$ be a bipartite graph  and let $T$ be a tree on an independent set $U$ of $H$. For 
vertices $r, u\in  U$,  $T_r$ denotes the  tree $T$ 
rooted at $r$, and $\Gamma (u; T_r )= (x_v, x(\Lambda_v))_{v\in C}$, where $C$ 
is the component of $T \setminus \{u\}$ containing  the root $r$. 
That is, $\Gamma(u;T_r)$ is a vector-valued random variable, the components of which are the pairs $(x_v, x(\Lambda_v)), v\in C$. 
If  $u=r$, 
then $C=\emptyset$,  $\Gamma (r; T_r)=\emptyset$ and hence   
$\ex[\, g \, | \Gamma (r;T_r)] = \ex[g]$ for all $g$. 
In particular, 
$\ex[ f_{r}| \Gamma (r; T_r)]=
\ex[ f_{r}] =1+\varepsilon$ by the same argument as in \eqref{JH10}. 
The denominator for the second normalization is 
 $\ex[ f_u | \Gamma (u; T_r)]$ for each $u\in U$. 
Note that, 
for $u\not= r$, 
\beq{JH12}  
 \! \!  \ex[ f_u | \Gamma (u; T_r)]
 = \ex[ f_u | x(  \cup_{v\in C} (\Lambda_u\cap \Lambda_v))]
 = \ex_{x_u} \Big[ \frac{\indi( x_u \sim x(  \cup_{v\in C} (\Lambda_u\cap \Lambda_v))) }{\rhoo \rho(x_u)^{| \cup_{v\in C} (\Lambda_u\cap \Lambda_v)|-1}}\Big]+ \varepsilon
 \eeq
 is a random variable depending only 
 on $x(  \cup_{v\in C} (\Lambda_u\cap \Lambda_v))$, where $C$ is the component of $T_r \setminus \{u \}$
 containing $r$ and the last expectation $\ex_{x_u}$ is taken  over 
the uniform random vertex $x_u$ of $G$. 
 
We now define, for a tree $T$ on an independent set of 
a bipartite graph $H$  and $r\in V(T)$, 
$$ 
 f_{_{_{\!  T_r}}} := 
\prod_{a\in V(T) } \frac{f_a}{\ex[f_a| \Gamma (a; T_r) ]}.
$$
For a bipartite graph $H$  with bipartition $A\cup B$, since 
$$
t_H(G)
=
\ex\left[
\prod_{a\in A } \indi\left(x_{a} \sim x(\Lambda_a )\right)\right]
=\lim_{\varepsilon\rightarrow 0}
\ex\left[
\prod_{a\in A } \left(\indi\left(x_{a} \sim x(\Lambda_a )\right)
+\varepsilon\rho(x_a)^{|\Lambda_a|}\right)\right]
$$
by, e.g., the dominated convergence theorem,  
it suffices to show that 
$$
\ex\left[
\prod_{a\in A } \left(\indi\left(x_{a} \sim x(\Lambda_a )\right)
+\varepsilon\rho(x_a)^{|\Lambda_a|}\right)\right]
\geq \rhoo^{|E(H)|},
$$
which is equivalent to
\beq{JHmain1}
\ex\left[
f_{_{\! T_r}}  
\prod_{a\in A}
\Big( \rhoo \rho(x_a)^{|\Lambda_a|-1} \ex[f_a| \Gamma (a; T_r) ]
\Big)\right]\geq \rhoo^{|E(H)|},
\eeq
provided $T$ is a tree on $A$ and $r
\in A$. 

Recall that  an independent set $U$ of a bipartite graph 
$H$ is 
tree-arrangeable if there is a tree $T$ on $U$ such that 
\beq{JHta}
\Lambda_{u}\cap\Lambda_{v}=\bigcap_{w\in P}\Lambda_{w}
~~~\mbox{for every path $P$ in $T$ connecting  $u$ and  $v$,}\eeq
and that a bipartite graph 
$H$  is tree-arrangeable 
if there exists a bipartition $A\cup B$ of $H$ such that $A$ is tree-arrangeable.

\mn 

The main lemma in this section  is 

\begin{LEM} \label{JHmain}
Suppose an independent set $U$ of a bipartite graph $H$ is  $T$-arrangeable for a tree
$T$  on $U$. 
Then, $f_{T_r}$ is root-invariant, that is,
$f_{T_r}=f_{T_s}$ for all $r,s \in U$.
Moreover, 
for  $u\in U$ and 
a random variable $g=g(x_u, x(\Lambda_u))$ determined by $x_u$ and $x(\Lambda_u)$, 
$$ 
 \ex [ gf_{T_r}  ] =\frac{\ex[gf_u ]}{1+\varepsilon},  
$$
regardless of the choice of $r\in U$. 
 In particular,  $\ex[f_{T_r}]=1 $.  
\end{LEM} 
\mn
Once this  is established, one may prove Theorem \ref{thm:main_thm} using Jensen's inequality for logarithmic functions, as we have seen above. 

\mn

We first prove the following lemma. 
 
\begin{LEM}\label{JHprop} Suppose  $U$ is an independent set of  a bipartite graph $H$
and is $T$-arrangeable for a tree $T$ on $U$. 
Then the following hold.

\begin{enumerate}[{\em (i)}]
\item If $S$ is a subtree  of $T$, then $V(S)$ is $S$-arrangeable. 

\item  For each $u\in U$, each component $C$ of $T\setminus \{u\}$
and the vertex $u^*$ in $C$ adjacent to $u$, we have that  
\beq{JHequiv}\bigcup_{v\in C} (\Lambda_u \cap \Lambda_v)
=\Lambda_u \cap \Lambda_{u^*}. \eeq

\item For distinct vertices $u,r\in U$, 
let $u_{_r}$ be the parent 
of $u$ in the rooted tree $T_r$, or equivalently, 
$u_{_r}$ be the vertex adjacent to $u$ in the path in $T$ connecting $u$ and $r$. Then 
$$
\ex[f_u|\Gamma(u;T_r)]=\ex[f_u|x(\Lambda_u \cap\Lambda_{u_{_r}})]=
\ex_{x_u} \Big[ \frac{\indi( x_u \sim x(  \Lambda_u\cap \Lambda_{u_{_r}})) }{\rhoo \rho(x_u)^{| \Lambda_u\cap \Lambda_{u_{_r}}|-1}}\Big]+ \varepsilon
$$
where the expectation $\ex_{x_u}$ is taken over the uniform random vertex $x_u$ of $G$. 
 
\end{enumerate} 
\end{LEM}

\begin{proof} 
Since  a path in $S$ is also a path in $T$, 
\eqref{JHta} holds for every path $P$ in $S$. Thus (i) holds.
For (ii), clearly $\Lambda_u \cap \Lambda_{u^*}
\subseteq \bigcup_{v\in C} (\Lambda_u \cap \Lambda_v)$. 
On the other hand, for every vertex $v\in C$, the path
$P$  in $T$ connecting $u$ and $v$ must contain $u^*$ and hence 
$$\Lambda_u \cap \Lambda_v =
\bigcap_{w\in P} \Lambda_w \subseteq \Lambda_u \cap \Lambda_{u^*}. $$ 
Therefore, 
$$
\bigcup_{v\in C} (\Lambda_u \cap \Lambda_v)
\subseteq \Lambda_u \cap \Lambda_{u^*}.
$$
The equalities in (iii) follows from \eqref{JH12} and (ii)
as $u^* = u_{_r}$ in this case. 
\end{proof}

\mn 
{\bf Remark}. It is not difficult to show that 
 \eqref{JHequiv} is a necessary and sufficient condition for the $T$-arrangeability.

\mn 
\begin{proof}[Proof of Lemma \ref{JHmain}.]
For the first part, it is enough to show that $f_{_{\! T_r}} = f_{_{\! T_s}}$ for  all adjacent pairs $r,s$ in $T$. 
Suppose $r,s$  are adjacent in $T$.
If $u\not= r,s$, then 
$\ex[ f_u |\Gamma (u; T_r) ] =\ex[ f_u |\Gamma (u; T_s) ] $  since $s$ and $r$ are in the same component of $T\setminus \{u\}$. Thus, $\ex[f_r| \Gamma (r; T_r)]
=\ex[f_s| \Gamma (s; T_s)]=1+\varepsilon$ implies that
$$ 
\frac{f_{_{\! T_s}}}{f_{_{\! T_r}}}
=\frac{\ex[f_s| \Gamma(s; T_r)]}{\ex[f_r| \Gamma(r; T_s)]}. $$
As $s$ and $r$ are adjacent in $T$, (iii) of Lemma \ref{JHprop} gives
\[
\ex[f_s|\Gamma(s;T_r)]= \ex_{x_s} \Big[ \frac{\indi( x_s \sim x(  \Lambda_s\cap \Lambda_{r})) }{\rhoo \rho(x_s)^{| \Lambda_s\cap \Lambda_{r}|-1}}\Big]+ \varepsilon 
= \ex_{x_r} \Big[ \frac{\indi( x_r \sim x(  \Lambda_s\cap \Lambda_{r})) }{\rhoo \rho(x_r)^{| \Lambda_s\cap \Lambda_{r}|-1}}\Big]+ \varepsilon=\ex[f_r|\Gamma(r;T_s)]
.\]
Therefore,
$$ 
\frac{f_{_{\! T_r}}}{f_{_{\! T_s}}}
=\frac{\ex[f_s| \Gamma(s; T_r)]}{\ex[f_r| \Gamma(r; T_s)]}=1. $$

For the second part, we first have that 
$$ 
 \ex [ gf_{T_r}  ] =\ex  [ gf_{T_u}]
 . 
$$
If $|V(T)|=|U|=1$, then 
$\ex  [ gf_{T_u}]= \ex [ g \frac{f_u}{\ex[f_u]}] = \frac{ \ex [ g f_u]}{1+\varepsilon}$, as desired. 
Suppose $|V(T)|=|U| \geq 2$. Then,  
for a leaf $\ell$ of $T$ other than $u$ and the  tree 
$S =T \setminus \{\ell\}$, the set $U\setminus \{\ell\}$ is $S$-arrangeable 
by (i) of Lemma \ref{JHprop}, and, for $v\in U\setminus \{\ell\}$ with $v\not=u$, it follows from   (iii) of Lemma \ref{JHprop} that 
$$
\ex[ f_v | \Gamma(v; S_u) ] =
\ex[ f_v | x(\Lambda_v \cap \Lambda_{v_u}) ]
= \ex[ f_v | \Gamma(v; T_u) ], $$
where $v_u$ is the parent of $v$ in $S_u$, as   $v_u$ is also the parent of $v$ in $T_u$. Therefore,
$\ex[ f_u| \Gamma(u; T_u)] =\ex[ f_u| \Gamma(u; S_u)] $,
and
$$ f_{T_u} = \frac{ f_{S_u} f_\ell}{\ex[ f_\ell| \Gamma(\ell; T_u)]}. $$
Hence
$$
 \ex\left[ g f_{T_u }\right]=
 \ex\bb[ 
\ex\bb[\frac{g f_{S_u}f_\ell}
     {\ex[f_\ell|\Gamma(\ell;T_u)] }\bb\vert  \Gamma(\ell; T_u)\bb] \bb]=
 \ex\bb[ g f_{S_u}
\ex\bb[\frac{f_\ell}
     {\ex[f_\ell| \Gamma(\ell; T_u)] }\bb\vert  \Gamma(\ell; T_u)\bb] \bb]=
 \ex\left[ g f_{S_u }\right],
$$
as both of $g$ and $f_{S_u}$ are determined by 
 $ \Gamma(\ell; T_u)$. Keep deleting vertices  of $T_u$ by the same way, we eventually have 
 $$ 
\ex  [ gf_{T_r}]=\ex [ g f_{T_u }]=\ex \bb[ g \frac{f_u}{\ex[f_u]}\bb]
=\frac{\ex [ g f_u]}{1+\varepsilon} ,$$
as desired. By taking $g\equiv 1$, we have that 
$$\ex  [ f_{T_r}]
=\frac{\ex [ f_u]}{1+\varepsilon}=1.\qedhere $$
\end{proof}

\old{\begin{proof}
For the other direction, we use the induction 
on the length $|P|$ of paths $P$, where the length of a path is the number of edges in it. 
For a path of length $1$, there is nothing to check. 
Suppose 
$$ \bigcap_{u\in P} \Lambda_u = \Lambda_{u_{_1}}
\cap \Lambda_{u_{_t}} $$
for all paths $P=u_{_1}, u_{_2}, ..., u_{_{t}}$ in $T$, $t\geq 2$.
If $P=u_{_1}, u_{_2}, ..., u_{_{t+1}}$,  then by the induction hypothesis 
$$ \bigcap_{u\in P} \Lambda_u = \Lambda_{u_{_1}}
\cap \Lambda_{u_{_t}} \cap \Lambda_{u_{_{t+1}}}. $$
For the connected component $C$ of $T\setminus \{ u_{_{t+1}}\}$ containing $u_{_1}$, \eqref{JHequiv} gives
$$
\bigcup_{v\in C} (\Lambda_{u_{_{t+1}}} \cap \Lambda_v)
=\Lambda_{u_{_{t+1}}} \cap \Lambda_{u_{_{t}}}.
$$
Since 
$ \Lambda_{u_{_{1}}} \cap \Lambda_{u_{_{t+1}}}
\subseteq
\bigcup_{v\in C} (\Lambda_{u_{_{t+1}}} \cap \Lambda_v)$,
we have that $$ \Lambda_{u_{_{1}}} \cap \Lambda_{u_{_{t+1}}}
\subseteq  \Lambda_{u_{_{t}}} \cap \Lambda_{u_{_{t+1}}}
$$
and hence 
$$ \Lambda_{u_{_{1}}} \cap \Lambda_{u_{_{t+1}}}
=  \Lambda_{u_{_{1}}} 
\cap \Lambda_{u_{_{t}}} \cap \Lambda_{u_{_{t+1}}}
=\bigcap_{u\in P} \Lambda_u . 
$$
\end{proof} }

\mn

\mn

We are now ready to prove  Theorem \ref{thm:main_thm}:
\begingroup
\def\thetheorem{\ref{thm:main_thm}}
\begin{theorem} (Restated) 
If a bipartite graph $H$ is tree-arrangeable, then $H$ has Sidorenko's property.
\end{theorem}
\addtocounter{theorem}{-1}
\endgroup

\begin{proof} Let $H$ be a bipartite graph with bipartition $A\cup B$ and let $A$ be $T$-arrangeable for a tree $T$ on $A$. As seen earlier in \eqref{JHmain1}, it suffices to show that, for a fixed vertex $r\in A$,  
$$
\ex\left[
f_{_{\! T_r }}  
\prod_{a\in A}
\Big( \rhoo \rho(x_a)^{|\Lambda_a|-1} \ex[f_a| \Gamma (a; T_r) ]
\Big)\right]\geq \rhoo^{|E(H)|}.
$$
Since $\ex[ f_{T_r}] =1$, Jensen's inequality gives
\bean 
\ln \ex\bb[
f_{_{\! T_r }}  
\prod_{a\in A}
\Big( \rhoo \rho(x_a)^{|\Lambda_a|-1} \ex[f_a| \Gamma (a; T_r) ]
\Big)\bb]
&\geq&  \ex\bb[
f_{_{\! T_r }}  
\ln \prod_{a\in A}
\Big( \rhoo \rho(x_a)^{|\Lambda_a|-1} \ex[f_a| \Gamma (a; T_r) ]
\bb) \bb] 
\eean 
The right hand side is 
$$ 
|A| \ex  [ f_{T_r} \ln \rhoo ] + \sum_{a\in A}(|\Lambda_a|-1)
\ex[ f_{T_r} \ln \rho (x_a)] +
\sum_{a\in A} \ex\bb[
f_{_{\! T_r }}  
\ln  \ex[f_a| \Gamma (a; T_r) ]
 \bb] .
$$

First, as $\ex  [ f_{T_r} ]=1$,
\beq{JHfirst} 
\ex  [ f_{T_r} \ln \rhoo ]= \ln\rhoo
\eeq	
Second, since  $\ln \rho (x_a)$ is determined 
by $x_a$, Lemma \ref{JHmain} 
together with the same argument used in \eqref{JH10} gives 
$$ \ex[ f_{T_r} \ln \rho (x_a)]
=\frac{\ex[ f_{a} \ln \rho (x_a)]}{1+\eps}
= \rhoo^{-1} \ex[ \rho (x_a)\ln \rho (x_a)]. $$
Jensen's inequality further gives 
\beq{JHsecond}
\ex[ f_{T_r} \ln \rho (x_a)]
=\rhoo^{-1} \ex[ \rho (x_a)\ln \rho (x_a)]
\geq   \rhoo^{-1} \ex[ \rho (x_a)]\ln \ex[\rho (x_a)] 
=\ln \rhoo. 
\eeq

Third, as $\ex[f_a| \Gamma (a; T_r) ]$ is determined by 
$x(a)$ and $x(\Lambda_a)$, Lemma \ref{JHmain} yields 
\bean 
\ex\bb[
f_{_{\! T_r }}  
\ln  \ex[f_a| \Gamma (a; T_r) ]
 \bb]& =&\frac{1}{1+\varepsilon}
 \ex\bb[
f_{a}  
\ln  \ex[f_a| \Gamma (a; T_r) ]
 \bb] \\
 &=& \frac{1}{1+\varepsilon}\ex\bb[\ex\bb[ 
f_{a}  
\ln  \ex[f_a| \Gamma (a; T_r)] \bb| \Gamma (a; T_r) \bb]
\bb] \\
&=& \frac{1}{1+\varepsilon}\ex\bb[\ex[f_a| \Gamma (a; T_r)]   
\ln  \ex[f_a| \Gamma (a; T_r)]  \bb].
\eean 
Applying  Jensen's inequality for the convex function 
$z\ln z$, and using $\ex\bb[\ex[f_a| \Gamma (a; T_r)] \bb]=\ex[f_a] =1+\eps$, we have that 
\bea{JHthird} 
\ex\bb[
f_{_{\! T_r }}  
\ln  \ex[f_a| \Gamma (a; T_r) ]
 \bb]
 &\geq&\frac{1}{1+\varepsilon} 
   \ex\bb[\ex[f_a| \Gamma (a; T_r)] \bb] \ln \ex\bb[\ex[f_a| \Gamma (a; T_r)] \bb] \nonumber \\
   & =& \ln (1+\eps)\geq 0. 
\eea 
Combining \eqref{JHfirst}-\eqref{JHthird},
we have that 
$$ \ln \ex\bb[
f_{_{\! T_r }}  
\prod_{a\in A}
\Big( \rhoo \rho(x_a)^{|\Lambda_a|-1} \ex[f_a| \Gamma (a; T_r) ]\Big)\bb]
\geq 
|A|\ln \rhoo+ \sum_{a\in A} |(\Lambda_a|-1)\ln \rhoo = |E(H)| \ln \rhoo, $$
or equivalently, 
$$
\ex\bb[
f_{_{\! T_r }}  
\prod_{a\in A}
\Big( \rhoo \rho(x_a)^{|\Lambda_a|-1} \ex[f_a| \Gamma (a; T_r) ]
\Big)\bb]
\geq \rhoo^{|E(H)|},$$
as desired. 
\end{proof}

\section{Cartesian products}\label{sec:recursive}

Recall that the Cartesian product $H_1 \sq H_2$ of two graphs $H_1$ and $H_2$
is defined as the graph on $V(H_1)\times V(H_2)$ 
such that two vertices $(u_1,u_2)$ and $(v_1,v_2)$ are adjacent if and only if 
either (i) $u_1$ and $v_1$ are adjacent in $H_1$ and $u_2=v_2$,
or (ii) $u_2$ and $v_2$ are adjacent in $H_2$ and $u_1=v_1$.
In this section we prove Theorem \ref{thm:box_product}, which 
is  restated for reader's convenience.

\begingroup
\def\thetheorem{\ref{thm:box_product}}
\begin{theorem} (Restated) 
If $T$ is a tree and $H$ is a bipartite graph having Sidorenko's property, 
then $T \sq H$ also has Sidorenko's property.
\end{theorem}
\addtocounter{theorem}{-1}
\endgroup

The following alternative description of the graph $T \sq H$ 
provides insight to the proof of Theorem \ref{thm:box_product}.
Let $T_1, \cdots,T_{|V(H)|}$ be vertex-disjoint copies of the graph $T$. For each edge $\{a,b\}$ of $H$, place an edge
between the copy of each $v\in V(T)$ in $T_a$ and $T_b$
so that $T_a$ and $T_b$ together form a copy of $T \sq K_2$. It is not too difficult to check
that the resulting graph is $T \sq H$.

We wish to count the number of homomorphisms
from $T \sq H$ to a given graph $G$, through
counting the number of homomorphisms from
$H$ to an auxiliary graph constructed from $G$.
For each vertex $v$ of $H$, there 
exists a copy $T_v$ of $T$ in $T \sq H$ over the 
vertices $V(T) \times \{v\}$. 
Moreover, as seen above, for each edge $e = \{v,w\}$
of $H$, the two copies of $T_v$ and $T_w$ form 
a copy of $T \sq K_2$ in $T \sq H$ (see Figure \ref{TH}). 
Thus, a copy of $T \sq K_2$ in $G$ needs to be contracted into 
an edge in the desired auxiliary graph of $G$.
This motivates the following definition
of the operation $\psi_T$ on $G$.

\begin{figure}
 \centering
    \includegraphics[width=0.7\textwidth]{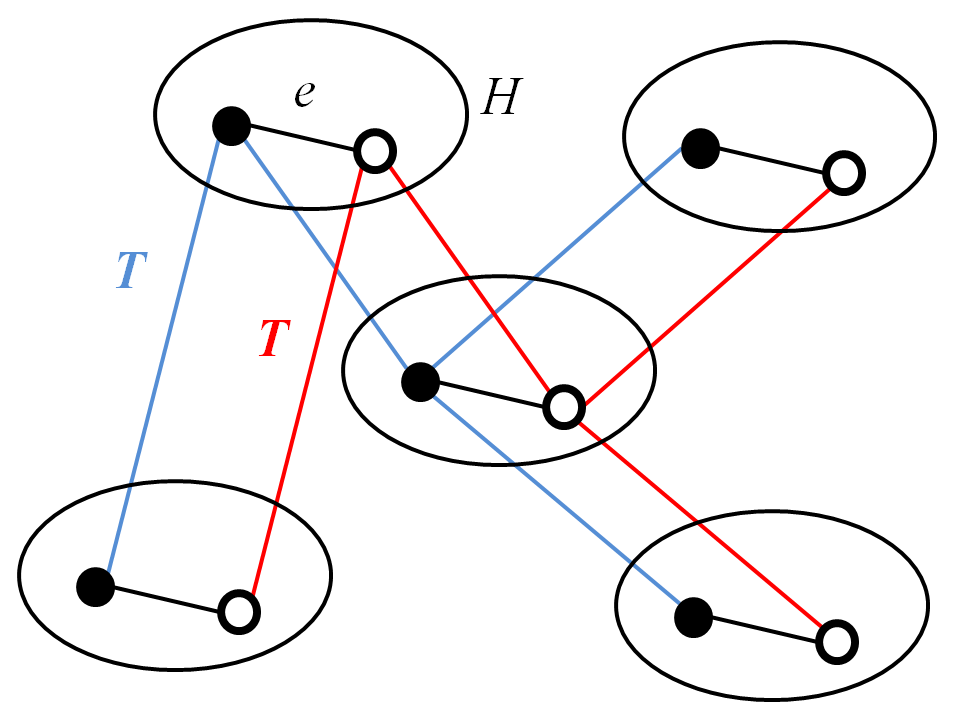}
    \caption{ $T \mbox{
    {\tiny $\!\! \square \,$}} K_2$ in 
    $T \mbox{
    {\tiny $\!\!  \square \,$}} H$.}\label{TH}
\end{figure}

\begin{DEF}
For given graphs $T$ and $G$, let $\psi_T(G)$ be the graph
with vertex set $\Hom(T,G)$ such that two vertices 
$h_1, h_2 \in \Hom(T,G)$ are adjacent if and only if
$h_1(v)$ and $h_2(v)$ are adjacent in $G$ for all $v \in V(T)$.
\end{DEF}

The observation above essentially is
equivalent to saying that a copy of $T \sq H$ in $G$
can be mapped to a copy of $H$ in $\psi_T(G)$, and
the following lemma formalizes this intuition.

\begin{LEM}\label{lem:1-1corresp}
For all graphs $T$, $H$ and $G$,
there exists a one-to-one correspondence between
$\Hom(T \sq H, G)$ and $\Hom(H, \psi_T(G))$.  
In particular,
\[
	|\Hom(T \sq H, G)| = |\Hom(H, \psi_T(G))|.
\]
\end{LEM} 
\begin{proof}
We will define $\xi : \Hom(T \sq H, G) \rightarrow \Hom(H, \psi_T(G))$
and $\varphi : \Hom(H, \psi_T(G)) \rightarrow \Hom(T \sq H, G) $
such that $\xi \circ \varphi = id$.

For a given $h \in \Hom(T \sq H, G)$, for 
each $v \in V(H)$, define $h_v : V(T) \rightarrow V(G)$ as
$h_v(w) = h(w,v)$ for each $w \in V(T)$. 
Whenever
$w,w' \in V(T)$ are adjacent, the vertices
$h_v(w) = h(w,v)$ and $h_v(w') = h(w',v)$
are adjacent. Thus $h_v \in \Hom(T, G)$
for all $v \in V(H)$.
Moreover, 
if $v, v'$ are adjacent vertices of $H$, then
$h_v(w) = h(w,v)$ and $h_{v'}(w) = h(w,v')$ are adjacent, 
and thus $h_v$ and $h_v'$ are adjacent in $\psi_T(G)$. 
Hence if we let $\xi(h) \,:\, V(H) \rightarrow \Hom(T, G)$ 
be defined by $\xi(h)(v) = h_v$, then $\xi$ is a map from
$\Hom(T \sq H, G)$ to $\Hom(H, \psi_T(G))$.

On the other hand, given a map $g \in \Hom(H, \psi_T(G))$, define
$\varphi(g) : V(T) \times V(H) \rightarrow V(G)$ as
$\varphi(g)(w,v) = g(v)(w)$ for each $v \in V(H)$ and $w \in V(T)$.
We first prove that $\varphi(g) \in \Hom(T \sq H, G)$.
For edges of the form $\{(w,v), (w',v)\}$,
$\varphi(g)(w,v) = g(v)(w)$ and $\varphi(g)(w',v) = g(v)(w')$ are adjacent
since $g(v) \in V(\psi_T(G)) = \Hom(T, G)$.
For edges of the form $\{(w, v), (w, v')\}$, we have
$\varphi(g)(w,v) = g(v)(w)$ and $\varphi(g)(w,v') = g(v')(w)$,
and these two vertices are adjacent in $G$ since
$g(v)$ and $g(w)$ are adjacent in $\psi_T(G)$. Hence we established
that $\varphi(G) \in \Hom(T \sq H, G)$.

It suffices to prove the $\xi \circ \varphi = id$. 
This follows from the fact that
for $h \in \Hom(T \sq H, G)$, $v \in V(H)$ and $w \in V(T)$,
\[
	\left((\varphi \circ \xi)(h)\right)(w,v) = h_v(w) = h(w,v),
\]
for the map $h_v$ defined as above.
\end{proof}

By Lemma \ref{lem:1-1corresp}, we can now estimate the 
size of $\Hom(T \sq H, G)$ 
through estimating the size of $\Hom(H, \psi_T(G))$, 
where Sidorenko's property of $H$ provides 
a lower bound on the size of $\Hom(H, \psi_T(G))$. 
We can use this idea to show the simplest case of 
Theorem~\ref{thm:box_product}, i.e., when $T=K_2$.
Here we give a full proof of this simple case, as 
the result will be used in the proof of Theorem \ref{thm:box_product}.
 
\begin{THM} \label{thm:simple_case}
If $H$ is a bipartite graph having Sidorenko's property, 
then $K_2 \sq H$ has Sidorenko's property.
\end{THM}
\begin{proof}
Let $G$ be a given graph and put $\psi(G) = \psi_{K_2}(G)$ for simplicity.
By Lemma~\ref{lem:1-1corresp} and the fact that $H$ has
Sidorenko's property, we have
\begin{align} \label{eq:corresp}
	|\Hom(K_2 \sq H, G)| &= |\Hom(H, \psi(G))|  \nonumber \\
		&\ge |V(\psi(G))|^{|V(H)|} \left( \frac{|\Hom(K_2, \psi(G))|}{|V(\psi(G))|^2} \right)^{|E(H)|} \nonumber \\
		&= |V(\psi(G))|^{|V(H)| - 2|E(H)|} |\Hom(K_2, \psi(G))|^{|E(H)|}. 
\end{align}
We have 
\[
	|V(\psi(G))| = |\Hom(K_2, G)| = |V(G)|^2 t_{K_2}(G).
\]
On the other hand, by Lemma \ref{lem:1-1corresp} with
$H = K_2$, we have
\[
	|\Hom(K_2, \psi(G))| = |\Hom(K_2 \sq K_2, G)|,
\]
where since $K_2 \sq K_2$ is isomorphic to $C_4$, by 
Sidorenko's property of $C_4$, we have
\[
	|\Hom(K_2 \sq K_2, G)| = |V(G)|^4 t_{C_4}(G) \ge |V(G)|^4 (t_{K_2}(G))^4.
\]
Therefore in \eqref{eq:corresp}, we get
\begin{align*}
	|\Hom(K_2 \sq H, G)| &\ge \Big(|V(G)|^2 t_{k_2}(G)\Big)^{|V(H)|-2|E(H)|}
			\cdot \Big(|V(G)| t_{K_2}(G)\Big)^{4|E(H)|} \\
	&= |V(G)|^{2|V(H)|} t_{K_2}(G)^{|V(H)| + 2|E(H)|}.
\end{align*}
Since $|V(K_2 \sq H)| = 2|V(H)|$ and $|E(K_2 \sq H)| = 2|E(H)| + |V(H)|$, 
we deduce that $K_2 \sq H$ has Sidorenko's property.
\end{proof}

If one attempts to use the same idea as in the proof of 
Theorem \ref{thm:simple_case} to prove Theorem \ref{thm:box_product}
for general graphs $T$ other than $K_2$,
then the inequality corresponding to \eqref{eq:corresp} will be
\begin{align*} 
	|\Hom(T \sq H, G)| 
	\ge
	|V(\psi_T(G))|^{|V(H)| - 2|E(H)|} |\Hom(K_2, \psi_T(G))|^{|E(H)|}. 
\end{align*}
Thus we need estimates on
$|V(\psi_T(G))|=|\Hom(T, G)|$ and 
$|\Hom(K_2, \psi_T(G))| = |\Hom(K_2 \sq T, G)|$.
If $T$ has Sidorenko's property, then $K_2 \sq T$ also has
Sidorenko's property by Theorem \ref{thm:simple_case}. Hence in 
this case we have lower bound estimates on  both $|V(\psi_T(G))|$ and 
$|\Hom(K_2, \psi_T(G))|$. Unfortunately, these bounds do not transfer 
to a lower bound on $|\Hom(T \sq H, G)|$, since such a lower bound 
requires an upper bound
on $|V(\psi_T(G))|$ if $|V(H)| - 2|E(H)| < 0$.

We solve this problem when $T$ is a tree, through the following lemma 
asserting that it suffices to consider graphs $G$ with
bounded maximum degree.

\begin{LEM} \label{lem:sido_maxdegree}
A bipartite graph $H$ has Sidorenko's property
if and only if for all graphs $G$ with maximum degree
at most $\frac{4|E(G)|}{|V(G)|}$,
\begin{align*} 
	t_{H}(G) \ge t_{K_2}(G)^{|E(H)|}.
\end{align*}
\end{LEM}

We also need the following lemma.
We omit the proof, which is based on tensor products of graphs. 
One may refer to Remark 2 of \cite{Sidorenko9192} (English version) for more details. 

\begin{LEM} \label{lem:tensor}
Let $H$ be  a bipartite graph. If there exists a constant $c$ depending
only on $H$ such that 
\[
t_{H}(G)\ge  c (t_{K_{2}}(G))^{|E(H)|}
~~~\mbox{for all graphs $G$,} 
\]
then $H$ has Sidorenko's property.
\end{LEM}

\begin{proof}[Proof of Lemma \ref{lem:sido_maxdegree}]
We may assume that $H$ has no isolated vertex, as adding 
an isolated vertex to a graph does not affect the value of $t_H(G)$
and $|E(H)|$.

Suppose that $H$ is a bipartite graph satisfying the given condition, and
let $G$ be an arbitrary graph (not necessarily satisfying
the maximum degree condition).

Let $\Delta = \frac{2|E(G)|}{|V(G)|}$, and 
let $G'$ be a graph obtained from $G$ by the following 
process. Fix an ordering of the vertices of $G$,
and take vertices $v$ one at a time 
according to the ordering.
Replace $v$ with $t = \lceil \frac{\deg(v)}{\Delta} \rceil$
vertices $v_1, \cdots, v_t$ and choose the neighbors
of these new vertices so that
(i) $N(v_i) \subseteq N(v)$, (ii) $N(v_i) \cap N(v_j) = \emptyset$
for all distinct pairs $i,j$, and
(iii) $\deg(v_i) \le \Delta$ for all $i$. 
Note that such a choice exists, as one can greedily
assign the neighbors of $v$ to the vertices $v_i$ under the given constraints.
Further note that during this process, 
$\deg(v)$ remains the same until $v$ is replaced, 
and the number of edges always remains the same as $|E(G)|$. 

Define a function $\pi:V(G')\ra V(G)$ as $\pi(v_i) = v$ for all $i$.
Since $\pi$ is a homomorphism from $G'$ to $G$,
we obtain a map $\phi: \Hom(H,G')\ra\Hom(H,G)$ 
such that $\varphi(h) := \pi \circ h$.
Further note that for an adjacent pair of vertices $v,w \in V(G)$, 
there exists a unique choice of $v' \in \pi^{-1}(v)$ and
$w' \in \pi^{-1}(w)$ such that $v'$ and $w'$ are adjacent in $G'$.
Therefore if $\pi \circ h_1 = \pi \circ h_2$ 
for some $h_1, h_2 \in \Hom(H,G')$, then for each edge $\{x,y\}$ of $H$,
we must have $h_1(x) = h_2(x)$ and $h_1(y) = h_2(y)$.
Since $H$ has no isolated vertex, we see that
$h_1(x) = h_2(x)$ for all $x \in V(H)$, i.e. $h_1 = h_2$. This implies that
our map $\varphi$ from  $\Hom(H,G')$ to $\Hom(H,G)$ is an injection.
Therefore, $|\Hom(H,G)| \ge |\Hom(H,G')|$.

The graph $G'$ has the same
number of edges as the graph $G$, and the number of
vertices is at most
\begin{align*}
	|V(G')| &= \sum_{v \in V(G)} \left\lceil \frac{\deg(v)}{\Delta} \right\rceil \\
			&\le |V(G)| + \sum_{v \in V(G)} \frac{\deg(v)}{\Delta}
			= |V(G)| + \frac{2|E(G)|}{\Delta} = 2|V(G)|.
\end{align*}
Combining this with the fact $|E(G')|=|E(G)|$, 
it follows that $G'$ has maximum degree 
$\Delta = \frac{2|E(G)|}{|V(G)|} \le \frac{4|E(G')|}{|V(G')|}$.
Hence $G'$ satisfies the given maximum degree condition, so
\begin{align*}
	|\Hom(H,G)| 
		\ge\,& |\Hom(H, G')| 
		\ge |V(G')|^{|V(H)|} \left(\frac{2|E(G')|}{|V(G')|^2} \right)^{|E(H)|}  \\
		\ge\,& 2^{|V(H)|-2|E(H)|} |V(G)|^{|V(H)|} \left(\frac{2|E(G)|}{|V(G)|^2} \right)^{|E(H)|}.
\end{align*}
By Lemma \ref{lem:tensor}, this concludes the proof.
\end{proof}

We are now ready to prove Theorem \ref{thm:box_product}.
As mentioned above, the proof follows the same line as of
the proof of Theorem \ref{thm:simple_case}, and uses 
Theorem \ref{thm:simple_case} as an ingredient.

\begin{proof}[Proof of Theorem \ref{thm:box_product}]
We may assume that $H$ has no isolated vertex, as adding 
an isolated vertex to a graph does not affect the value of $t_H(G)$
and $|E(H)|$.

Let $T$ be a tree with $\tau$ vertices, 
and let $G$ be a given graph. By Lemma \ref{lem:sido_maxdegree},
we may assume that $G$ has maximum degree
at most $\frac{4|E(G)|}{|V(G)|} = 2|V(G)|t_{K_2}(G)$.
By Lemma \ref{lem:1-1corresp} and the fact that $H$ has
Sidorenko's property, we have
\begin{align} \label{eq:corresp_general}
	|\Hom(T \sq H, G)| &= |\Hom(H, \psi_T(G))|  \nonumber \\
		&\ge |V(\psi_T(G))|^{|V(H)|} \left( \frac{|\Hom(K_2, \psi_T(G))|}{|V(\psi_T(G))|^2} \right)^{|E(H)|} \nonumber \\
		&= |V(\psi_T(G))|^{|V(H)| - 2|E(H)|} |\Hom(K_2, \psi_T(G))|^{|E(H)|}. 
\end{align}
Recall that $V(\psi_T(G)) = \Hom(T,G)$.
We can construct an element in $\Hom(T,G)$ by starting from an
arbitrary vertex of $T$, defining its image in $V(G)$, 
and then extending the homomorphism one vertex
at a time. By the condition on the maximum degree of $G$, we thus have
\begin{align}
	|V(\psi_T(G))| &= |\Hom(T, G)| \le |V(G)| \Big( 2|V(G)|t_{K_2}(G) \Big)^{\tau-1} \nonumber \\
			&= 2^{\tau-1}|V(G)|^{\tau} t_{K_2}(G)^{\tau-1}.  \label{eq:v_psi_t}
\end{align}
On the other hand, by Lemma \ref{lem:1-1corresp} with
$H = K_2$ and Theorem \ref{thm:simple_case}, we have
\begin{align} \label{eq:e_psi_t}
	|\Hom(K_2, \psi_T(G))| = |\Hom(T \sq K_2, G)| \ge |V(G)|^{2\tau} t_{K_2}(G)^{3\tau - 2}.
\end{align}
Since $H$ has no isolated vertex, we have $|V(H)| \le 2|E(H)|$, and
thus in \eqref{eq:corresp_general}, we may use the bounds from
\eqref{eq:v_psi_t} and \eqref{eq:e_psi_t} to obtain
\begin{align*}
	&|\Hom(T \sq H, G)|  \\
	\ge\,& \Big( 2^{\tau-1}|V(G)|^{\tau} 	t_{K_2}(G)^{\tau-1} \Big)^{|V(H)|-2|E(H)|}
			\Big(|V(G)|^{2\tau} t_{K_2}(G)^{3\tau - 2}\Big)^{|E(H)|} \\
	=\,& 2^{(\tau-1)(|V(H)|-2|E(H)|)} \cdot |V(G)|^{\tau|V(H)|}t_{K_2}(G)^{(\tau-1)|V(H)| + \tau |E(H)|}.
\end{align*}
Since $|V(T \sq H)| = \tau|V(H)|$ and $|E(T \sq H)| = (\tau-1)|V(H)| + \tau |E(H)|$, 
by Lemma \ref{lem:tensor}, we deduce that $T \sq H$ has Sidorenko's property.
\end{proof}

Since an arbitrary $d$-dimensional grid can be obtained from
the Cartesian product of $d$ paths, we obtain the following corollary.

\begin{COR}
For all $d \ge 1$, all $d$-dimensional grids have Sidorenko's property.
\end{COR}


\section{Concluding Remarks}
\label{sec:conclude}

In this section, we will say more about tree-arrangeability
and possible extensions of Theorem \ref{thm:box_product}.
First, we will provide a simple description of tree-arrangeability 
in terms of the vertices with maximal neighbors.
Second, we will explain how the tree-arrangeability is related to tree decompositions and Markov Random Field.
We conclude by proposing a couple of open questions related 
to Cartesian products that may illuminate 
a way to attack Sidorenko's conjecture.

\medskip

\noindent {\bf Tree-arrangeability and vertices with maximal neighborhood}.
To see whether a bipartite graph $H$ with bipartition $A\cup B$ is tree-arrangeable,
it suffices to consider only the vertices in $A$ 
whose neighborhoods are maximal with respect to inclusion.
A subset $U$ of $A$ is called {\em neighbor covering} if for each $a\in A$, there exists
$u\in U$ such that $\Lambda_a\subseteq\Lambda_u$.
If a neighbor covering set $U$ is $T$-arrangeable 
for a tree $T$ on $U$, then 
the tree on $A$ obtained by 
adding each $a\in A\setminus U$ to 
$T$ as a leaf adjacent to $u\in U$ with 
$\Lambda_a\subseteq\Lambda_u$
(if more than one such $u$ exists, then choose arbitrary 
one among them) makes $A$ tree-arrangeable.
Hence $H$ is tree-arrangeable if and only if there exists a 
neighbor covering set $U \subseteq A$ that is tree-arrangeable.
The cases when there exists a neighbor covering set of size 
one or two were discussed in the introduction.

\medskip

\noindent {\bf Tree-arrangeability and tree decompositions}.
Tree-arrangeability can be alternatively defined using tree decompositions.
A tree decomposition of a graph $H$, introduced by Halin \cite{Halin}
and developed by Robertson and Seymour \cite{RobertSeymour}, 
is a pair $(\mathcal{F}, T)$ of  a 
family $\mathcal{F}$  of vertex subsets and  a tree $T$ with vertex set $\mathcal{F}$
satisfying
\begin{enumerate}
\item $\bigcup_{X\in\mathcal{F}}X=V(H)$, 
\item for each $\{v,w\} \in E(H)$, there exists a set $X \in \mathcal{F}$ such that
$v, w \in X$, and
\item for $X,Y,Z\in \mathcal{F}$, $X\cap Y\subseteq Z$ 
whenever $Z$ lies on the path from $X$ to $Y$ in $T$.
\end{enumerate}
It is straightforward to check that a bipartite graph 
$H$ with bipartition $A\cup B$ is
tree-arrangeable if and only if there exists a tree 
decomposition of $H$ with 
$\mathcal{F}=\{\Lambda_a\cup\{a\} \,|\, a\in A\}$.

\medskip
\noindent {\bf Markov Random Field}.
Tree-arrangeability and the functions $f_u$ defined in 
Section \ref{sec:proof_mainthm} are also closely related to 
Markov Random Field theory.

A sequence of random variables 
$(y_v)_{v\in V(G)}$ is said to be a {\em Markov Random Field} 
with respect to a graph $G$ if for each $S\subseteq V(G)$
that makes $G\setminus S$ disconnected, whenever $C_1$ and $C_2$
are the vertex sets of distinct components of $G \setminus S$, the pair of sequences
of random variables $(y_{v})_{v\in C_1}$ and $(y_{v})_{v\in C_2}$
is independent, 
conditioned on $(y_v )_{v\in S}$.\footnote{There are a few non-equivalent definitions of 
a Markov Random Field. Here we state the most general definition.}
Lemma \ref{JHprop} (iii) shows that if a bipartite graph $H$ 
with bipartition $A \cup B$ is tree-arrangeable with a 
tree $T$ on $A$, 
then $(f_v)_{v \in A}$ for the random variables $f_v$ defined 
in Section~\ref{sec:proof_mainthm} is a Markov Random Field with respect 
to $T$.
It would be interesting to further investigate the connection between the theory of
Markov Random Fields and Sidorenko's conjecture.

\medskip

\noindent {\bf Extension of Cartesian product to non-bipartite graphs}.
For a given (not necessarily bipartite) graph $H$, 
define a bipartite graph $\phi(H)$ as follows: 
The bipartition of $\phi(H)$ consists of 
two disjoint copies of $V(H)$. 
Two vertices in distinct parts  are adjacent in $\phi(H)$ 
if they are copies of the same vertex in $H$, or two adjacent 
vertices in $H$. In particular, 
$\phi(H)$ has $2|E(H)|+|V(H)|$ edges.

\begin{figure}[h!]
 \centering
    \includegraphics[width=0.8\textwidth]{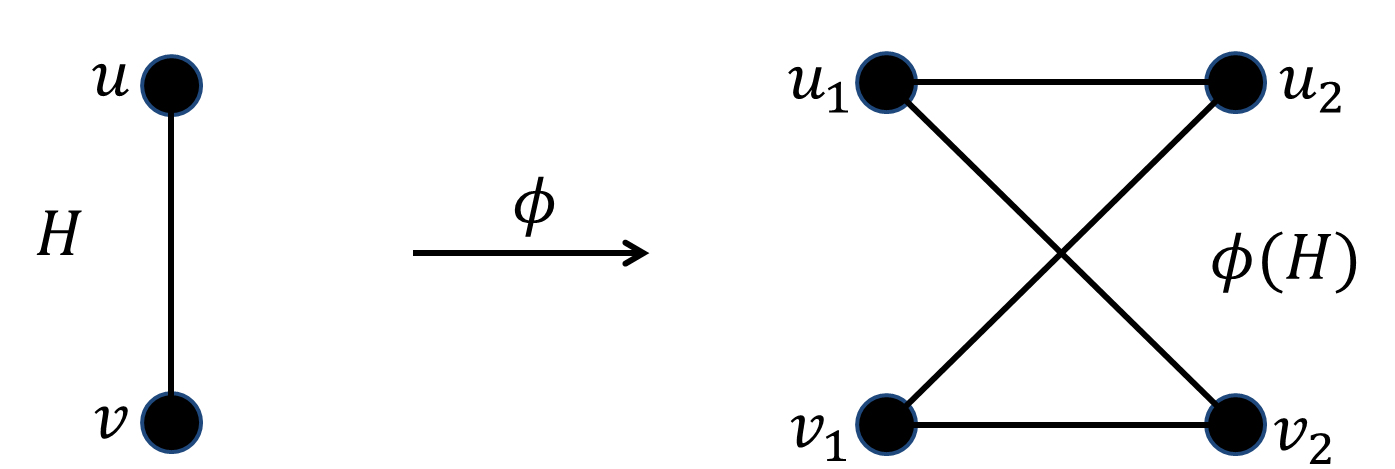}
   \caption{Blow-up via $\phi$.}\label{blowup}
  \end{figure}

It is not too difficult to see that for bipartite graphs $H$,
we have $\phi(H) = K_2 \sq H$. Hence the operation $\phi$ is more
restricted than Cartesian products when considering bipartite graphs.
However, the operation $\phi$ has
the advantage of being applicable to non-bipartite graphs.
For example, since $\phi(K_k) = K_{k,k}$,
we know that $\phi(K_k)$ has Sidorenko's property
for all $k \ge 2$. Thus $\phi(H)$ may have Sidorenko's property
even if $H$ is a non-bipartite graph. Also note that
$\phi(C_5)$ is $K_{5,5} \setminus C_{10}$ which is the minimal
bipartite graph unknown to satisfy Sidorenko's conjecture.
The operation $\phi$ provides many interesting graphs for 
which Sidorenko's conjecture is not known to be true.

We conclude the paper with some open problems regarding 
the operator $\phi$.
We believe that the family $\{\phi(C_{2k+1})\}_{k \ge 1}$ can be
an interesting starting point in further studying
Sidorenko's conjecture. The only known graph to have Sidorenko's
property in this family is $C_3$.
\begin{QUES}
Does there exist an integer $k \ge 2$ such that $\phi(C_{2k+1})$ has Sidorenko's property?
\end{QUES}

Since $\phi(H) = K_2 \sq H$ for bipartite graphs, 
Theorem \ref{thm:box_product} implies that $\phi(H)$ has Sidorenko's property
as long as $H$ does. Hence $\phi(H)$ is `more likely' to have
Sidorenko's property than $H$.
For example, since $\phi(K_r)=K_{r,r}$ for  integers $r\ge 1$, 
we know that $\phi(K_r)$ has Sidorenko's property, 
while  $K_r$ is not even a bipartite graph. 
(Recall that a graph $H$ with odd cycles cannot satisfy Sidorenko's property as $t_H (G) =0$ for bipartite graphs $G$).
Thus, the following question may be posed. 

\begin{QUES}
For a (not necessarily bipartite)  graph $H$, does there exist a non-negative integer $k=k_{_{\! H}}$ such that 
$\phi^k(H)$ has Sidorenko's property?
\end{QUES}
If Sidorenko's conjecture is true, then it certainly implies 
that the answers to the questions above are both yes.
Even if Sidorenko's conjecture turns out to be false,
it is possible that the answers 
to the questions are positive.

\medskip

\noindent {\bf Acknowledgement.} Part of this work was done
while Choongbum Lee and Joonkyung Lee were visiting 
Jeong Han Kim at KIAS.
We would like to thank David Conlon for 
his useful comments and suggestion to consider the grid graphs, and
Jacob Fox for providing the idea behind Lemma \ref{lem:sido_maxdegree}.
We would also like to thank the anonymous referee for the valuable comments.


\begin{thebibliography}{100}
\bibitem{BenPeres}
I. Benjamini and Y. Peres, A correlation inequality for tree-indexed Markov chains, in \emph{Seminar on
Stochastic Processes}, Proc. Semin., Los Angeles/CA (USA) 1991, Prog. Probab. \textbf{29}, 1992, 7--14.

\bibitem{BlRo}
G. Blakley and P. Roy,
H\"{o}lder type inequality for symmetrical matrices with non-negative entries,
\emph{Proc. Amer. Math. Soc.} (1965) \textbf{16}, 1244-1245.


\bibitem{CoFoSu}D. Conlon, J. Fox, and B. Sudakov, An approximate
version of Sidorenko's conjecture, \emph{Geom. Funct. Anal.} \textbf{20}
(2010), 1354-1366.

\bibitem{DaLeMo}
R. Daudel, R. Lefebvre, and C. Moser, \emph{Quantum chemistry: Methods and applications},
Interscience Publishers, New York-London (1959).

\bibitem{ErSi}P. Erd\H{o}s and M. Simonovits, Cube-supersaturated
graphs and related problems, in \emph{Progress in graph theory }(Waterloo,
Ont., 1982), Academic Press, Toronto, ON, 1984, 203-218.

\bibitem{ErStone}
P. Erd\H{o}s and A. H. Stone, 
On the structure of linear graphs, in \emph{Bull. Amer. Math. Soc.} \textbf{52}(12) (1946), 1087-1091.

\bibitem{FKG}
C. M. Fortuin, P.W. Kasteleyn, and J. Ginibre, 
Correlation inequalities on some partially ordered sets, 
\emph{Comm. Math. Phys.} \textbf{22} (1971), 89-03.

\bibitem{StatPhy}
Hans-Otto Georgii, \emph{Gibbs measures and phase transitions},
\textbf{Walter de Gruyter} (1988).

\bibitem{Halin}
R. Halin, S-functions for graphs, \emph{J. Geom.} \textbf{8} (1976), 171-186.

\bibitem{Hatami}H. Hatami, Graph norms and Sidorenko's conjecture,
\emph{Israel J. Math.}\textbf{\emph{175}} (2010), 125-150.

\bibitem{London}
D. London, Two inequalities in nonnegative symmetric matrices. Pac. J. Math. 
\emph{16}, 
515-536 (1966)


\bibitem{LoSi}L. Lov\'asz and M. Simonovits, On the number of complete
subgraphs in a graph II, Studies in pure mathematics, Birkh\"auser
(1983), 459-495.



\bibitem{MuSm}
H. Mulholland and C. Smith,
An inequality arising in genetical theory,
\emph{Amer. Math. Monthly} (1959) \textbf{66}, 673--683.

\bibitem{Nikiforov}V. Nikiforov, The number of cliques in graphs
of given order and size, \emph{Trans. Amer. Math. Soc.} \textbf{363}
(2011), 1599-1618.

\bibitem{PmanPeres}
R. Pemantle and Y. Peres, Domination between trees and application to an explosion problem,
\emph{Ann. Probab.} \textbf{22} (1994), 180-194

\bibitem{Razborov}A. Razborov, On the minimal density of triangles
in graphs, \emph{Combin. Probab. Comput. }\textbf{17} (2008), 603-618.

\bibitem{Reiher}C. Reiher, The clique density theorem, arXiv:1212.2454
{[}math.CO{]}.

\bibitem{RobertSeymour}N. Robertson and P. D. Seymour, Graph minors III: Planar tree-width, 
\emph{J. Combin. Theory Ser. B.} \textbf{36(1)} (1984), 49-64.

\bibitem{Sidorenko86}A. F. Sidorenko, Extremal problems in graph
theory and inequalities in functional analysis (in Russian), in \emph{Proceedings
of the Soviet Seminar on Discrete Mathematics and its applications
(in Russian)}, ed. Lupanov, O.B., Moscow, Moscow State University
(1986), 99-105.

\bibitem{Sidorenko9192}A. F. Sidorenko, Inequalities for functionals
generated by bipartite graphs, \emph{Diskretnaya Matematika} \textbf{3}
(1991), 50-65 (in Russian), \emph{Discrete Math. Appl.}\textbf{ 2
}(1992), 489-504 (in English).

\bibitem{Sidorenko93}A. F. Sidorenko, A correlation inequality for
bipartite graphs, \emph{Graphs Combin. }\textbf{9} (1993), 201-204.

\bibitem{Sidorenko94}A. F. Sidorenko, An analytic approach to extremal
problems for graphs and hypergraphs,in \emph{Extremal problems for
finite sets} (Visegr\'ad, 1991), Bolyai Soc. Math. Stud., 3, J\'anos
Bolyai Math. Soc., Budapest, 1994, 423-455.

\bibitem{Simonovits}M. Simonovits, Extremal graph problems, degenerate
extremal problems and super-saturated graphs, in \emph{Progress in
graph theory }(Waterloo, Ont., 1982), Academic Press, Toronto, ON,
1984, 419-437.


\bibitem{ImgProcess}
Stan Z. Li. \emph{Markov Random Field Modeling in Image Analysis}, \textbf{Springer} (2001).

\bibitem{Stell}
G. Stell,
Generating functionals and graphs,
in \emph{Graph Theory and Theoretical Physics}, Academic Press, London (1967), 281--300.


\bibitem{LiSz}J.X. Li and B. Szegedy, On the logarithimic calculus
and Sidorenko's conjecture, arXiv:1107.1153 {[}math.CO{]}.

\bibitem{Turan}P. Tur\'an, Eine Extremalaufgabe aus der Graphentheorie,
\emph{Mat. Fiz. Lapok} \textbf{48} (1941), 436\textendash{}452.
\end{thebibliography}
\end{document}